\documentclass[12pt]{amsart}
\usepackage{amssymb, xcolor}

\newtheorem{thm}{Theorem}[section]
\newtheorem{lem}[thm]{Lemma}

\theoremstyle{definition}

\theoremstyle{remark}
\newtheorem{rmk}[thm]{Remark}

\numberwithin{equation}{section}

\newcommand{\z}{{\mathbb Z}}
\newcommand{\q}{{\mathbb Q}}
\newcommand{\rank}{\text{rank}}

\newcommand{\floor}[1]{\left\lfloor #1 \right\rfloor}

\usepackage{array}

\begin{document}


\author{BYEONG MOON KIM, Myung-Hwan Kim and Dayoon Park}
\address{Department of Mathematics, Kangnung National University, Kangnung, 210-702, Korea}
\email{kbm@kangnung.ac.kr}
\thanks{}

\address{Department of Mathematical Sciences and Research Institute of Mathematics,
Seoul National University, Seoul 08826, Korea}
\email{mhkimath@snu.ac.kr}
\thanks{This work was supported by the Institute for Information and Communications
Technology Promotion (IITP) Grant through the Korean Government (MSIT),
(Development of lattice-based post-quantum public-key cryptographic schemes),
under Grant 2017-0-00616.}

\address{Department of Mathematics, The University of Hong Kong, Hong Kong}
\email{pdy1016@hku.hk}
\thanks{}






\title[real quadratic fields admitting
universal lattices] {real quadratic fields admitting universal
lattices of rank 7}

\begin{abstract}
In this paper, we will prove that if $d$ is sufficiently large
square-free positive rational integer, then there is no
positive definite universal
quadratic $\mathcal{O}_F$-lattice of rank 7 where $F=\mathbb
Q(\sqrt{d})$.
\end{abstract}

\maketitle

\section{Introduction}

\vskip 0.3cm

Let $F$ be a totally real number
field and let $\mathcal O_F$ be the ring of algebraic integers of
$F$. A {\it quadratic $\mathcal O_F$-lattice} is a finitely
generated $\mathcal O_F$-module $L$ equipped with a quadratic map
$Q:L \rightarrow \mathcal O_F$, for which
$$B(x,y):=\frac{Q(x+y)-Q(x)-Q(y)}{2} \quad \forall x,y\in L$$ is symmetric
bilinear form on $L$ and $B(x,x)=Q(x)$. $L$ is said to be {\it
classic} if $B(x,y)\in \mathcal O_F$ for all $x,y\in L$, and {\it
non-classic}, otherwise. \vskip 0.1cm

A quadratic $\mathcal O_F$-lattice
$L$ is said to be {\it positive definite} if $Q(x)\in \mathcal
O_F^+$ for all $x\in L\setminus \{0\}$, where $\mathcal
O_F^+(\subset \mathcal O_F)$ is the set of all totally positive
algebraic integers of $F$. A positive definite quadratic
$\mathcal{O}_F$-lattice $L$ is said to be {\it universal} if it
represents every element of $\mathcal O_F^+$, i.e., if
$Q(L)=\mathcal O_F^+ \cup \{0\}$. For example, the famous Maass'
three square theorem \cite{m} states that the sum of three squares
over $\mathcal O_F$ for $F=\mathbb Q(\sqrt{5})$, which corresponds
to the positive definite quadratic $\mathcal O_F$-lattice
$L=\mathcal O_F \times \mathcal O_F \times \mathcal O_F$ equipped
with the quadratic map $Q(x_1,x_2,x_3)=x_1^2+x_2^2+x_3^2$, is
universal. \vskip 0.1cm

Universality is one of the most
fascinating topics in the arithmetic theory of quadratic lattices
and has been studied in depth for a long time. In 1993, Conway and
Schneeberger announced a celebrated finiteness theorem on universal
quadratic $\mathbb Z$-lattices, known as the {\it 15-Theorem} (see
\cite{CS}), which states that the universality of a classic positive
definite quadratic $\mathbb Z$-lattice is determined by the
representability of positive rational integers $1,2,\dots,15$.
Bhargava and Hanke \cite{290} later proved the {\it 290-theorem}
which states that a non-classic positive definite quadratic $\mathbb
Z$-lattice is universal if it represents all positive rational
integers up to $290$. \vskip 0.1cm

According to Hsia, Kitaoka and
Kneser \cite{HKK}, we may conclude that for any totally positive
real number field $F$, there always exist positive definite
universal $\mathcal {O}_F$-lattices. A very interesting and
meaningful question is what the minimal rank of positive definite
universal $\mathcal {O}_F$-lattices is for a given $F$. As an
example, it is well known that $4$ is the minimal rank of classic positive definite
universal $\mathcal {O}_F$-lattices if $F=\mathbb
Q$ and $3$ if $F=\mathbb Q(\sqrt{2}),\mathbb Q(\sqrt{3}),\mathbb
Q(\sqrt{5})$. In the former case, all classic positive definite universal
$\mathbb Z$-lattices of rank 4 are listed by Bhargava \cite{CS}. In the latter case, all classic positive
definite universal $\mathcal {O}_F$-lattices of rank 3 are listed by
Chan, M.-H. Kim and Raghavan \cite{CKR}. 
There are some results that minimal rank of positive definite universal quadratic form of a totally real number field could be arbitrarily large.
That was shown that for a given $n \in \mathbb N$, there are infinitely many real quadratic number fields \cite{BK}, multiquadratic number fields \cite{KS}, and cubic fields \cite{Y} that do not admit positive definite universal quadratic forms (possibly non-classic) of rank $n$ by Kala, Kala-Svoboda, and Yatsyna, respectively.
By covering all of these results, Kala \cite{K} recently proved that for a given $n,d \in \mathbb N$ with $2|d$ or $3|d$, there are infinitely many real number fields $F$ with $[F:\mathbb Q]=d$ that do not admit positive definite universal quadratic forms (possibly non-classic) of rank $n$.

\vskip 0.1cm

On the other hand, it is quite
natural to ask for a given positive integer $n$, how many totally
real number fields $F$ are there which admit a positive definite
universal $\mathcal {O}_F$-lattice of rank $n$. Kitaoka conjectured
that there are only finitely many totally real number fields $F$
which admit a positive definite universal quadratic form in 3
variables (i.e., free quadratic lattice of rank 3) over $\mathcal
O_F$, in a private seminar. The conjecture was partly answered by
Chan, M.-H. Kim and Raghavan \cite{CKR}. They proved that among real
quadratic number fields, only $F=\mathbb Q(\sqrt{2}), \mathbb
Q(\sqrt{3})$ and $\mathbb Q(\sqrt{5})$ admit a classic positive definite
universal $\mathcal {O}_F$-lattice of rank 3. Kala and Yatsyna
\cite{KY} showed that $\mathbb Q$ and $\mathbb Q(\sqrt{5})$ are the only totally real
number fields which admit a
positive definite universal
quadratic form in 3 (besides 4 and 5, too) variables whose coefficients are all rational
integers. 
Kr\'{a}sensk\'{y}, Tinkov\'{a}, Zemkov\'{a}
\cite{KTZ} proved that there is no biquadratic totally real number
field which admits a classic positive
definite universal quadratic form in 3 variables. All of these
are supporting evidences of Kitaoka's Conjecture. \vskip 0.1cm

In 2000, B. M. Kim \cite{octonary}
proved that there are infinitely many real quadratic number fields
$F$ which admit a positive definite universal free quadratic
$\mathcal{O}_F$-lattice of rank 8. But it was not known whether 8 is
the minimal rank with this property. He proved earlier \cite{d.s}
that there are only finitely many real quadratic number fields $F$
which admit a positive definite diagonal ('diagonal' means having no cross terms) universal free quadratic
$\mathcal{O}_F$-lattice of rank 7. \vskip 0.1cm

In this paper, we prove that there
exist only finitely many real quadratic number fields $F$ which
admit a positive definite universal quadratic
$\mathcal{O}_F$-lattice of $\rank 8$. This implies that 8 is indeed the minimal
rank $n$ that guarantees the infinitude of real quadratic number
fields $F$ which admit a positive definite universal quadratic
$\mathcal{O}_F$-lattice of rank $n$. 
 \vskip 5em

\section{Lemmas}
\vskip 0.3cm

\begin{lem} \label{0} Let
$L$ be a positive definite $\mathcal O_F$-lattice and let $\begin{pmatrix}B(v_i,
v_j)\end{pmatrix}_{n\times n}$ be the Gram-matrix of given vectors
$v_1, v_2, \cdots, v_n \in L$. Then

\noindent $(1)$ \ $\begin{pmatrix}B(v_i, v_j)\end{pmatrix}$ is
positive semi-definite\,;

\noindent $(2)$ \ $\det \begin{pmatrix}B(v_i, v_j)\end{pmatrix}>0$
if and only if $v_1, v_2, \cdots, v_n$ are linearly independent\,;

\noindent $(3)$ \ $\det \begin{pmatrix}B(v_i, v_j)\end{pmatrix}=0$
if and only if $v_1, v_2, \cdots, v_n$ are linearly dependent.
\end{lem}

We skip the proof of Lemma 2.1,
which is quite well-known. We need the following technical lemma.

\begin{lem}\label{1}Let $A:=(a_{ij})
\in M_k(\mathbb R),
B:=(b_{ij}):=\begin{pmatrix} B_1 & B_2 \\ B_3 & B_4 \end{pmatrix}
\in M_{k+s}(\mathbb R)$ where $B_1 \in M_k(\mathbb R)$, $B_2,B_3^T
\in M_{k \times s}(\mathbb R)$ and $B_4 \in M_s(\mathbb R)$ with
$|a_{ij}|, |b_{ij}| \le N$ where $N>1$. Then for $x>0$,
$$\begin{array}{l}
\det(A)\cdot\det(B_4)\cdot x^k-k! \cdot (k+s)!\cdot
N^{k+s}\cdot (x^{k-1}+\cdots+x+1)\\
<\det(A)\cdot\det(B_4)\cdot x^k-N^{k+s}\cdot \sum \limits_{l=0}^{k-1}
\frac{k!}{(k-l)!\cdot l!} \cdot \frac{k!\cdot(k+s-l)!}{(k-l)!}\cdot x^l\\
 \le \det \left(
\begin{pmatrix} A & O_{k\times s} \\ O_{s \times k} & O_{s \times s}
\end{pmatrix}x +
B \right)\\
\le \det(A)\cdot\det(B_4)\cdot x^k+N^{k+s}\cdot \sum \limits_{l=0}^{k-1}
\frac{k!}{(k-l)!\cdot l!} \cdot \frac{k!\cdot(k+s-l)!}{(k-l)!}\cdot x^l\\
<\det(A)\cdot\det(B_4)\cdot x^k+k! \cdot (k+s)!\cdot N^{k+s}\cdot
(x^{k-1}+\cdots+x+1)\end{array}$$
holds, where $O_{k\times s} \in
M_{k\times s}(\mathbb R), O_{s \times k}\in M_{s \times k}(\mathbb
R),$ and  $O_{s \times s}\in M_{s}(\mathbb R)$ are zero matrices and
$\det(B_4):=1$ when $s=0$.
\end{lem}

\begin{proof}
Note that for an $n\times n$ matrix $E=(e_{ij})$, its determinant is
defined by 
$$\det(E)=\sum_{\mathbf
i}sgn(\mathbf i) \cdot e_{1i_1}\cdots e_{ni_n}$$ where $\mathbf
i:=(i_1,\cdots,i_n)$ runs through
all $n$-permutations and $sgn(\mathbf i)$ is the parity of $\mathbf
i$. From the definition of determinant of matrix, 
by arranging the determinant as descending order with respect to $x$,
we have that
$$\det \left(
\begin{pmatrix} A & O_{k\times s} \\ O_{s \times k} & O_{s \times s}
\end{pmatrix}x +
B \right) =d_kx^k+d_{k-1}x^{k-1}+\cdots+d_1x+d_0$$ where
\begin{equation} \label{a_l}
d_l=\sum_{\mathbf i,\,\mathbf j} sgn(\mathbf i) \cdot sgn(\mathbf j)
\cdot a_{i_{1}j_{1}}\cdots a_{i_{l}j_{l}}b_{i_{l+1}j_{l+1}}\cdots
b_{i_{n}j_{n}}
\end{equation}
where $\mathbf i := (i_1,\cdots,i_{k+s})$ and $\mathbf j :=
(j_1,\cdots,j_{k+s})$ run through all
$(k+s)$-permutations with $i_1<\cdots<i_l \le k$,
$i_{l+1}<\cdots<i_{k+s}$, and $j_1,\cdots, j_l \le k$.
Since the number of permutations
$\mathbf i$ satisfying $i_1<\cdots<i_l \le k$ and
$i_{l+1}<\cdots<i_{k+s}$ is $\frac{k!}{(k-l)! \cdot l!}(<k!)$
and the number of permutations
$\mathbf j$ satisfying $j_1,\cdots, j_l \le k$ is $\frac{k!}{(k-l)! \cdot
l!}\cdot l! \cdot (k+s-l)!=\frac{(k+s-l)!k
!}{(k-l)!}=(k+s-l)\cdots(k+1-l)\cdot k!(\le (k+s)!)$, the total number of
terms of the summation of (\ref{a_l}) is $\frac{k!}{(k-l)! \cdot l!}
\cdot \frac{k ! \cdot (k+s-l)!}{(k-l)!}(< k! \cdot (k+s)!)$. And the
absolute value $|a_{i_{1}j_{1}}\cdots
a_{i_{l}j_{l}}b_{i_{l+1}j_{l+1}}\cdots b_{i_{n}j_{n}}|$ of each term
could not exceed $N^{k+s}$, so we have $|d_l| \le
\frac{k!}{(k-l)!\cdot l!} \cdot \frac{k!\cdot(k+s-l)!}{(k-l)!}\cdot
N^{k+s} < k! \cdot (k+s)! \cdot N^{k+s}$ for all $1 \le l \le k$.
Moreover, we may see $d_k=\det(A)\cdot\det(B_4)$. This completes the
proof.
\end{proof}
\vskip 0.3cm

From here on, we assume that $F=\mathbb Q(\sqrt{d})$ is a real
quadratic number field where $d>1$ is a square free positive
integer. 
We write the discriminant of $F=\mathbb Q(\sqrt{d})$ as $\Delta_d$, i.e.,
$$\Delta_d=\left\{\begin{array}{ll} 4d \quad &\mbox{when} \ \
d\equiv 2, 3 \pmod{4},\\ d \quad &\mbox{when} \ \ d\equiv
1\quad \pmod{4}.\end{array}\right.$$
For any $z=x+y\sqrt{d}\in F$ with $x,y\in\mathbb Q$, we
denote its conjugate by $\overline{z}:=x-y\sqrt{d}$ and put
$$\omega_d:=\left\{\begin{array}{ll} \sqrt{d} \quad &\mbox{if} \ \
d\equiv 2, 3 \pmod{4},\\ \frac{1+\sqrt{d}}{2} \quad &\mbox{if} \ \ d\equiv
1\quad \pmod{4}.\end{array}\right.$$ It is well known that every
algebraic integer $\alpha\in\mathcal{O}_F$ can be written in the form
$$\alpha=a+b\omega_d$$
where $a,b\in\mathbb Z$, namely, $\{1, \omega_d\} $ is a $\z$-basis of $\mathcal O_F$.

\vskip 0.3cm

\begin{lem}\label{2}
For $\alpha, \beta \in\mathcal{O}_F \setminus\mathbb Z$ with $\alpha \in
\mathcal{O}^+_F$, we have
${\rm tr}(\alpha)\ge
\sqrt{\Delta_d}$ \,and\,\ ${\rm tr}(\beta^2)\ge \frac{\Delta_d}{2}$.
\end{lem}

\begin{proof} 
We only prove only for $d \equiv 2,3 \pmod 4$.
One may easily see for $d \equiv 1 \pmod 4$ similarly with this proof.
Since $\alpha \in \mathcal{O}^+_F \setminus \mathbb Z$ is a form of
$a+b\sqrt d$ where $a, b\in \mathbb Z \setminus \{0\}$ with $a>|b \sqrt
d|$, we have that ${\rm tr}(\alpha)=2a>2|b \sqrt d|\ge 2\sqrt d$. Since
$\beta \in \mathcal{O}_F \setminus \mathbb Z$ is a form of $a+b\sqrt
d$ where $a, b(\not=0) \in \mathbb Z$, we have
$\beta^2=(a^2+b^2\cdot d)+ 2ab\sqrt d$ and so ${\rm
tr}(\beta^2)=2(a^2+b^2\cdot d) \ge 2d$.
 \end{proof}

\vskip 0.3cm

\begin{lem} \label{4}
Let $F=\mathbb Q(\sqrt{d})$ satisfying its discriminant
$$\Delta_d>4 \cdot 15^2.$$
Let
$L$ be a positive definite classic $\mathcal{O}_F$-lattice and let $$ S:=\{v
\in L \,|\, Q(v)\in \mathbb Z \text{ with } 1\le Q(v)\le 15\}.$$
Then the $\mathbb Z$-submodule $L_1$ of $L$ generated by $S$ is a
positive definite $\mathbb Z$-lattice, i.e., $B(L_1,L_1) \subseteq \z$.
\end{lem}

\begin{proof}
For all $v,w\in S$, $Q(v)Q(w)-B(v,w)^2$ is either $0$ or a
totally positive integer since $L$ is positive definite.
So we have that \[{\rm tr}(B(v,w)^2)\le{\rm
tr}(Q(v)Q(w))\le2\cdot15^2.\]
On the other hand, by Lemma \ref{2},
\[{\rm tr}(\beta^2)>2\cdot15^2\]
for $\beta \in \mathcal O_F \setminus \z$.
So we obtain that $B(v,w) \in \z$ for any $v,w \in S$,
which implies that $B(L_1,L_1) \subseteq \z$.
And the positive definiteness of $L_1$ follows from that of $L$.
This completes the proof.
\end{proof}
\vskip 0.3cm

\begin{rmk} \label{rmk2.5}
Let $F=\mathbb Q(\sqrt{d})$ satisfying its discriminant
$$\Delta_d>4 \cdot 15^2.$$ 
Let $L$ be a positive
definite universal $\mathcal{O}_F$-lattice.
By the universality of $L$, we may
take a vector $v_n \in L$ with $Q(v_n)=n$ for each integer $n$ from
$1$ up to $15$. Then $\begin{pmatrix} B(v_i,v_j)
\end{pmatrix}_{15\times15}$ is a symmetric positive semi-definite matrix in $M_{15 \times 15}(\z)$
by Lemma \ref{4}. According to the {\it 15-Theorem}, the
$\mathbb Z$-sublattice $L_1$ of $L$ generated by
$\{v_1,v_2,\dots,v_{15}\}$ is universal over $\z$. Since the minimal rank of a
positive definite universal $\mathbb Z$-lattice is $4$ (the fact that there is no ternary universal quadratic $\z$-lattice is easily induced in virtue of the Hilbert Reciprocity Law), we can find four linearly independent vectors
$v_1',v_2',v_3',v_4' \in \{v_1,v_2,\dots,v_{15}\}\subset L_1$ .
\end{rmk}

\begin{lem} \label{9}
Let $F=\mathbb Q(\sqrt{d})$ 
satisfying its discriminant
$$\Delta_d>(4\cdot 4!\cdot 4!\cdot15^4)^2,$$
and $L$ be a positive definite classic universal
$\mathcal{O}_F$-lattice. For $k=1,2,\dots, 15$, let $v_k \in L$ with
$$Q(v_k)=m_k+k \omega _d \in \mathcal O_F^+$$ where
$m_k=-\lfloor k\overline{\omega_d}\rfloor$, and let $L_2$ be the
$\mathcal{O}_F$-sublattice of $L$ generated by $\{v_1, v_2, \cdots ,
v_{15}\}$. Then $rank(L_2)\ge 4$.

Moreover, there are four
linearly independent vectors
$$v_1',\cdots,v_4' \in \{v_1,\cdots, v_{15}\}$$
satisfying
\begin{equation} \label{2.3}
\det
\begin{pmatrix}
(a_{ij}')
_{4 \times 4}
\end{pmatrix}>0
\end{equation}
where
$\begin{pmatrix}
B({v_i'},{v_j'})
\end{pmatrix}_{4 \times 4} = \sqrt{\Delta_d}\begin{pmatrix}
{a_{ij}'}
\end{pmatrix}_{4 \times 4}+
\begin{pmatrix}
{\epsilon_{ij}'}
\end{pmatrix}_{4 \times4}$
with
$$\left\{
\begin{array}{rl} \begin{pmatrix}
a_{ij}'
\end{pmatrix}_{4 \times 4} \in M_{4}(\mathbb Z) \quad
&\mbox{with} \ \ |{a_{ij}}'|\le15,\\
\begin{pmatrix}
\epsilon_{ij}'
\end{pmatrix}_{4 \times 4} \in M_{4}(\mathbb R)
&\mbox{with} \ \
|{\epsilon_{ij}}'|<1.\end{array}\right.$$
\end{lem}

\begin{proof}
From the fact that $L$ is positive definite, we have
\begin{equation}\label{B}
\begin{cases}
Q(v_i)\cdot Q(v_j)-B(v_i,v_j)^2 \ge 0\\
\overline{Q(v_i)}\cdot \overline{Q(v_j)} -\overline{B(v_i,v_j)}^2 \ge 0
\end{cases}
\end{equation}
for all $1 \le i,j \le 15$.
From (\ref{B}), we obtain that
\begin{equation}\label{B'}
\begin{cases}
B(v_i,v_j)^2 \le Q(v_i)\cdot Q(v_j)<(i\sqrt{\Delta_d}+1)(j\sqrt{\Delta_d}+1)\\
\overline{B(v_i,v_j)}^2 \le \overline{Q(v_i)}\cdot \overline{Q(v_j)}<1\cdot 1=1,
\end{cases}
\end{equation}
namely,
\begin{equation}\label{B''}
\begin{cases}
|B(v_i,v_j)|<\sqrt{(i\sqrt{\Delta_d}+1)(j\sqrt{\Delta_d}+1)}\\
|\overline{B(v_i,v_j)}|<1
\end{cases}
\end{equation}
for all $1\le i,j \le 15$. 
From the second condition of (\ref{B''}),
since $0\le \overline{B(v_i,v_j)}<1$ or $0 \le -\overline{ B(v_i,v_j)} <1$,
we have $$B(v_i,v_j) =(m_{a_{ij}}+a_{ij}\omega_d) \text{ or } -(m_{a_{ij}}+a_{ij}\omega_d)$$ for some $a_{ij}\in \mathbb Z $ where $m_{a_{ij}}=-\lfloor a_{ij}\overline{\omega_d}\rfloor$,
and from the first condition of (\ref{B''}), we have $$|a_{ij}|\le
\sqrt{ij} \le 15.$$ Since
$(m_{a_{ij}}+a_{ij}\omega_d)-a_{ij} \sqrt{\Delta_d}$ are in $[0,1]$, the
Gram-matrix of $L_2$ with respect to $\{v_1,v_2,\cdots, v_{15}\}$
may be written in the form
$$\begin{pmatrix}
B({v_i},{v_j})
\end{pmatrix}_{15 \times 15} = \sqrt{\Delta_d}\begin{pmatrix}
a_{ij}
\end{pmatrix}_{15 \times 15}+
\begin{pmatrix}
\epsilon_{ij}
\end{pmatrix}_{15 \times 15}$$
where $\left\{
\begin{array}{rl} \begin{pmatrix}
a_{ij}
\end{pmatrix}_{15 \times 15} \in M_{15}(\mathbb Z) \quad
&\mbox{with} \ \ |{a_{ij}}|\le15 \text{ and } a_{ii}=i\\
\begin{pmatrix}
\epsilon_{ij}
\end{pmatrix}_{15 \times 15} \in M_{15}(\mathbb R)
&\mbox{with} \ \ |{\epsilon_{ij}}|<1.\end{array}\right.$ \vskip
0.1cm

On the other hand, since $L$ is positive definite, the Gram-matrix
$$\begin{pmatrix}
B({v_i'},{v_j'})
\end{pmatrix}_{k \times k}=\sqrt{\Delta_d}\begin{pmatrix}
{a_{ij}'}
\end{pmatrix}_{k \times k}+
\begin{pmatrix}
{\epsilon_{ij}'}
\end{pmatrix}_{k \times k}$$
is positive semi-definite by Lemma
\ref{0} for any ${v_1'},\cdots, {v_k'} \in \{v_1, \cdots, v_{15}\}$
with $a_{ij}' \in \mathbb Z \cap [-15,15]$ and $\epsilon_{ij}'\in
[-1,1]$, which yields
$$\det\begin{pmatrix}
B({v_i'},{v_j}')_{k\times k}
\end{pmatrix}\ge0.$$
We may use Lemma \ref{1} with
$x=\sqrt{\Delta_d}$ and $|{a_{ij}'}|,|{\epsilon_{ij}'}|\le15=N$ to see that for
any ${v_1'},\cdots, {v_k'} \in \{v_1,v_2, \cdots, v_{15}\}$ for
$k=1,2,3,4$,
$$\det \left(\begin{pmatrix} {a_{ij}'}
\end{pmatrix}_{k \times k} \right) \ge 0.$$
For otherwise (i.e., if $\det \left(\begin{pmatrix} {a_{ij}'}
\end{pmatrix}_{k \times k}\right) \in \mathbb Z_{<0}$ for $k=1,2,3,4$), then
\begin{align*}
\det\begin{pmatrix}
B({v_i'},{v_j}')_{k\times k}
\end{pmatrix} & <\det(\begin{pmatrix} {a_{ij}'}
\end{pmatrix}_{k \times k})\sqrt{\Delta_d}^k+4!\cdot4!\cdot 15^k\cdot(\sqrt{\Delta_d}^{k-1}+\cdots+1)
\\
 & < -\sqrt{\Delta_d}^k+ 4!\cdot4!\cdot 15^k\cdot k\cdot\sqrt{\Delta_d}^{k-1}\\
& = \sqrt{\Delta_d}^{k-1}(-\sqrt{\Delta_d}+k \cdot 4!\cdot4!\cdot 15^k)<0
\end{align*}
holds.
This is a contradiction to the positive definiteness of $L$.
Therefore we obtain that for any ${v_1'},\cdots, {v_k'} \in
\{v_1,v_2, \cdots, v_{15}\}$ for $k=1,2,3,4$,
$$\begin{pmatrix} {a_{ij}'}
\end{pmatrix}_{k \times k}$$
is positive semi-definite symmetric matrix in $M_{k \times
k}(\mathbb Z)$. Since the minimal
rank of a positive definite universal $\mathbb Z$-lattice is $4$, by
the {\it 15-Theorem} \cite{CS}, we may take $v_1',v_2',v_3',v_4'
\in \{v_1,v_2, \cdots, v_{15}\}$ such that $\begin{pmatrix} a_{ij}'
\end{pmatrix}_{4 \times 4} \in M_4(\z)$ is positive definite, i.e.,
with $\det \begin{pmatrix}
(a_{ij}'
\end{pmatrix}_{4 \times 4})>0$.
As a brief illustration, take $v_1'=v_1$, $v_2'=v_2$ and $v_3'=v_3$ (resp. $v_3'=v_6$) when $a_{12}=\pm1$ (resp. $a_{12}= 0$).
And then $(a_{ij}')_{3\times 3} \in M_3(\mathbb Z)$ is a positive definite.
So we may take a positive integer $h \le 15$ which is not represented by $(a_{ij}')_{3\times 3}$ over $\z$.
By taking $v_4'=v_h$, we may construct a positive definite $(a_{ij}')_{4 \times 4} \in M_4(\z)$.
By applying Lemma \ref{1} again, we obtain
$$\det \begin{pmatrix}
 B(v_i',v_j')
_{4 \times 4}\end{pmatrix}>0,$$
yielding $v_1',v_2',v_3',v_4'$ are linearly independent by Lemma \ref{0}.
This completes the proof.
\end{proof}

\vskip 0.7cm

\section{Main Theorem}
\vskip 0.3cm
In this last section, we conclude that the rank of a positive definite classic
universal $\mathcal O_F$-lattice $L$ is greater than or equal to $8$ for
$F=\q (\sqrt{d})$ 
satisfying
\begin{equation} \label{d cond}
\Delta_d>(29524500000005)^2.
\end{equation}

\vskip 0.3cm

\begin{thm} \label{Main Thm}
For $F=\mathbb Q(\sqrt{d})$ satisfying (\ref{d cond}),
there is no positive definite
classic universal $\mathcal{O}_F$-lattice of rank $7$.\end{thm}

\begin{proof}
Let $L$ be a positive definite classic universal $\mathcal O_F$-lattice.
By Remark \ref{rmk2.5}, we may take four vectors $v_5,\cdots, v_8 \in L$ for which
$$\begin{pmatrix}
B({v_i},{v_j})
\end{pmatrix}_{5 \le i,j \le 8} \in M_4(\z)$$ is positive definite with
$|B({v_i},{v_j})| \le 15$ for all $5\le i,j \le 8$.
And by Lemma \ref{9}, we may take four vectors
$v_1,\cdots, v_4 \in L$ for which
$$\begin{pmatrix}
B({v_i},{v_j})
\end{pmatrix}_{1 \le i,j \le 4}= \sqrt{\Delta_d}\begin{pmatrix}
a_{ij}
\end{pmatrix}_{4 \times 4}+
\begin{pmatrix}
\epsilon_{ij}
\end{pmatrix}_{4 \times 4}$$
where $$\left\{
\begin{array}{rl} \begin{pmatrix}
a_{ij}
\end{pmatrix}_{4 \times 4} \in M_{4}(\mathbb Z) \quad
&\mbox{with} \ \ |{a_{ij}}|\le 15 \\
\begin{pmatrix} \epsilon_{ij}
\end{pmatrix}_{4 \times 4} \in M_{4}(\mathbb R)
&\mbox{with} \ \ |{\epsilon_{ij}}|<1\end{array}\right.$$ and $\det
\begin{pmatrix} a_{ij}
\end{pmatrix}_{4 \times 4} \ge 1.$ Now consider the Gram-matrix
$$\begin{pmatrix}
B({v_i},{v_j})
\end{pmatrix}_{1 \le i,j \le 8} \in M_8(\mathcal O_F).$$
We may write
$$\begin{pmatrix}
B({v_i},{v_j})\end{pmatrix}_{1 \le i,j \le 8}=
\sqrt{\Delta_d}
\begin{pmatrix}
A & O \\ O & O
\end{pmatrix}+
\begin{pmatrix}
D & C \\ C^T & B
\end{pmatrix}$$
where $A:=\begin{pmatrix}a_{ij}\end{pmatrix}_{1\le i,j \le 4},B:=
\begin{pmatrix}b_{ij}\end{pmatrix}_{5\le i,j \le 8}\in M_4(\z)$
are symmetric positive definite and the absolute values of the
entries of $A,B,C,D \in M_4(\mathbb R)$ are less than or equal to
$15$. For by Lemma \ref{4} and from the construction of
$v_1,v_2,v_3,v_4$, we obtain that
$A$ and $B$ are symmetric positive definite and the absolute values
of the entries of $A,B$ and $D$ are less than or equal to $15$. And
from the positive definiteness of $L$,
$$Q(v_i)Q(v_j)-B(v_i,v_j)^2=
 (\sqrt{\Delta_d} \cdot a_{ii}+d_{ii})\cdot b_{jj}-c_{ij}^2$$
where $1\le i\le 4$
and $5 \le j \le 8$, should be either totally positive or
zero for $C:=\begin{pmatrix}c_{ij}\end{pmatrix}_{1\le i \le 4, 5\le
j \le 8} \in M_4(\mathcal O_F)$. Note that
$$
\left\{\begin{array}{ll} 0< (\sqrt{\Delta_d} \cdot a_{ii}+d_{ii})\cdot b_{jj}
<  (15\sqrt{\Delta_d}+1)\cdot15  \\
0<\overline{(\sqrt{\Delta_d} \cdot a_{ii}+d_{ii})\cdot b_{jj}}<1\cdot 15  \end{array}\right.
$$
for $1 \le i \le 4$ and $5\le j \le 8$.
In order for  $\overline{(\sqrt{\Delta_d} \cdot a_{ii}+d_{ii})\cdot b_{jj}-c_{ij}^2}$
to be positive $|\overline{c_{ij}}|$ should be less than $4$, so the integer $c_{ij}$
should be of a form of $$-\floor{n_{ij} \cdot \overline{\omega_d}}+n_{ij} \cdot \omega_d+m_{ij}$$
for some $n_{ij} \in \mathbb Z$ and $m_{ij} \in \mathbb Z \cap [-4,3]$.
On the other hand, if $c_{ij}
= -\floor{n_{ij} \cdot \overline{\omega_d}}+n_{ij} \cdot \omega_d+m_{ij}
\in \mathcal O_F \setminus \mathbb Z$ (i.e., $n_{ij} \not=0$), then
$$\frac{\sqrt{\Delta_d}}{2} \ll |c_{ij}|$$ 
and hence
$(\sqrt{\Delta_d} \cdot a_{ii}+d_{ii})\cdot b_{jj}-c_{ij}^2<0$, which is a
contradiction to the positive definiteness of $L$. Thus $c_{ij}$
would be a rational integer in $[-4,3]$.

Using Lemma \ref{1}, we obtain
\begin{align*}
\det
\begin{pmatrix}
B({v_i},{v_j})\end{pmatrix}_{8 \times 8} & \ge \det(A)\det(D)(\sqrt{\Delta_d})^4
-15^8\sum \limits_{l=0}^3 \frac{4!\cdot 4!\cdot (8-l)!}{(4-l)! \cdot l!
\cdot (4-l)!} \cdot (\sqrt{\Delta_d})^l\\
 & \ge (\sqrt{\Delta_d})^4-15^8\sum \limits_{l=0}^3  \frac{4!\cdot 4!\cdot (8-l)!}{(4-l)!
 \cdot l! \cdot (4-l)!} \cdot (\sqrt{\Delta_d})^l>0
\end{align*}
because $x^4-15^8\sum \limits_{l=0}^3
\frac{4!\cdot 4!\cdot (8-l)!}{(4-l)! \cdot l! \cdot (4-l)!} \cdot
x^l>0$ for all $$x\ge
29524500000005.$$ These indicate that $v_1, v_2, \cdots, v_8$ are
linearly independent. Thus the rank of $L$ is greater than 7.
\end{proof}
\vskip 0.3em

\begin{thm} \label{Main Thm'} Let
$F=\mathbb Q(\sqrt{d})$ where $d$ is a square free positive integer
satisfying 
$$\Delta_d>(576283867731072000000005)^2.$$
Then there is no positive
definite non-classic universal $\mathcal{O}_F$-lattice of rank
7.\end{thm}
\begin{proof}
This theorem can be proved easily
by applying {\it 290-Theorem} \cite{290} instead of {\it 15-Theorem}
\cite{CS} in the proofs of Lemmas preceding Theorem \ref{Main Thm}.
\end{proof}


\vskip 1cm


\begin{thebibliography}{abcd}

\bibitem {290} M. Bhargava and J. Hanke, {\em Universal quadratic forms and the 290 Theorem},
Invent. Math., to appear.

\bibitem {BK} V. Blomer, V. Kala, {\em Number fields without universal n-ary quadratic forms}, Math. Proc. Cambridge Philos. Soc.
159 (2015), 239–252


\bibitem {CKR} W.K. Chan, M.-H. Kim, S. Raghavan,
{\em Ternary Universal Quadratic Forms
over Real Quadratic Fields}, Japanese J. Math. 22 (1996), 263-273.

\bibitem {CS} M. Bhargava, {\em On the Conway-Schneeberger fifteen theorem},
Contem. Math. 272 (2000), 27-38.

\bibitem {D} L.E. Dickson, {\em Quaternary quadratic forms representing all integers},
Amer. J. Math. 49 (1927), 39-56.

\bibitem {HKK}J.S. Hsia, Y. Kitaoka, M. Kneser, {\em Representation of positive definite
quadratic forms}, J. reine angew. Math. 301 (1978), 132-141.


\bibitem {d.s}  B.M. Kim, {\em Finiteness of real quadratic fields which admit
positive integral diagonal septanary universal forms}, Manuscr Math.
99 (1999), 181-184.

\bibitem {octonary} B.M. Kim, {\em Universal octonary diagonal forms over some
real quadratic fields}, Comment. Math. Helv. 75 (2000), 410-414.

\bibitem {KTZ} Jakub Kr\'{a}sensk\'{y}, Magdal\'{e}na Tinkov\'{a}, Krist\'{y}na Zemkov\'{a},
{\em There are no universal ternary quadratic forms over biquadratic fields},
Proc. Edinburgh. Math. Soc. 63 (2020), 861-912.

\bibitem {K} V. Kala, {\em Number fields without universal quadratic forms of small rank exist in most degrees}, preprint.

\bibitem {KS} V. Kala, J. Svoboda, {\em Universal quadratic forms over multiquadratic fields}, Ramanujan J. 48 (2019), 151–157

\bibitem {KY2} V. Kala, P. Yatsyna, {\em Universal quadratic forms and elements of small norm in real quadratic fields}, Bull. Aust. Math. Soc.
94 (2016), 7–14

\bibitem {KY} V. Kala, P. Yatsyna, {\em Lifting problem for universal quadratic forms}, preprint.

\bibitem {m} H. Maass, {\em \"Uber
die Darstellung total positiver Zahlen des K\"orpers $R(\sqrt 5)$
als Summe von drei Quadraten}, Abh. Math. Sem. Hamburg
\textbf{14}(1941), 185--191.

\bibitem {O} O.T. O'Meara, {\em Introduction to quadratic forms}, Springer Verlag,
New York (1973).
V. Kala, J. Svoboda, Universal quadratic forms over multiquadratic fields, Ramanujan J. 48 (2019), 151–157

\bibitem {Y} P. Yatsyna, {\em A lower bound for the rank of a universal quadratic form with integer coefficients in a totally real
field}, Comment. Math. Helvet. 94 (2019), 221–239

\end{thebibliography}
\end{document}